\documentclass[final, a4paper]{amsart}
\usepackage[english]{babel}
\usepackage{amsmath, amsfonts,amssymb,amsopn,amscd,amsthm,mathrsfs}
\usepackage{bm}
\usepackage{graphicx}
 \usepackage{enumerate}
 \usepackage{array}
 \usepackage{hyperref}
 \hypersetup{colorlinks=true,linkcolor=blue,citecolor=blue,linktoc=page}
\usepackage{showkeys}
\numberwithin{equation}{section}
\usepackage{fancyhdr}
\newcommand{\eq}{\begin{equation}}
\newcommand{\qe}{\end{equation}}

\newcommand{\E}{\mathbb{E}}
\newcommand{\N}{\mathbb{N}}
\newcommand{\R}{\mathbb{R}}
\newcommand{\p}{\mathbb{P}}
\theoremstyle{plain}

\newtheorem{thm}{Theorem}

\newtheorem{prop}[thm]{Proposition}
\newtheorem{cor}[thm]{Corollary}
\newtheorem{defn}{Definition}[section]
\theoremstyle{definition}

\theoremstyle{remark}
\newtheorem*{rem}{Remark}

\newtheorem{conjec}{Conjecture}

\usepackage{color}

\author{Kevin Tanguy \\ University of Angers, France}
\address{Kevin Tanguy is with the LAREMA (CNRS UMR 6093). Universit\'e d'Angers 49035 Angers, France.}

\begin{document}
\sloppy
\pagestyle{headings} 
\title{Talagrand inequality at second order and application to Boolean analysis}
\date{Note of \today}


\email{kevin.tanguy@univ-angers.fr}
\urladdr{http://perso.math.univ-toulouse.fr/ktanguy/}

\maketitle

\begin{abstract}
This note is concerned with an extension, at second order, of an inequality on the discrete cube $C_n=\{-1,1\}$ (equipped with the uniform measure) due to Talagrand (\cite{TalL1L2}). As an application, the main result of this note is a Theorem in the spirit of a famous result from Kahn, Kalai and Linial (cf. \cite{KKL}) concerning the influence of Boolean functions. The notion of the influence of a couple of coordinates $(i,j)\in\{1,\ldots,n\}^2$ is introduced in section \ref{two} and the following alternative is obtained : for any Boolean function $f\,:\, C_n\to \{0,1\}$, either there exists a coordinate with influence at least of order $(1/n)^{1/(1+\eta)}$, with $\, 0<\eta<1$ (independent of $f$ and $n$) or there exists a couple of coordinates $(i,j)\in\{1,\ldots,n\}^2$ with $i\neq j$, with influence at least of order $(\log n/n)^2$. In section \ref{four}, it is shown that this extension of Talagrand inequality can also be obtained, with minor modifications, for the standard Gaussian measure $\gamma_n$ on $\R^n$  ; the obtained inequality can be of independent interest. The arguments rely on interpolation methods by semigroup together with hypercontractive estimates. At the end of the article, some related open questions are presented.
\end{abstract}
\maketitle 
\section{Introduction}

The notion of influence of variables of Boolean functions has been extensively studied over the last twenty years with applications in various areas such as random graph theory, percolation theory and Gaussian geometry, (cf. e.g. the survey \cite{KS}). Now, let us introduce the setting of our work, for more details on the analysis of Boolean functions we refer the reader to \cite{Odo, GarStei}. Let $n\geq 1$ be and consider the discrete cube $C_n=\{-1,1\}^n$ equipped with the uniform measure $\mu^n$. The influence of the $i$-th coordinate of any function $f\,:C_n\,\to\{0,1\}$ is defined as follow.

\begin{defn}
Consider some function $f\,:C_n\,\to\{0,1\}$. For any $i\in\{1,\ldots,n\}$, the influence of the $i$-th coordinate is given by

\begin{equation}\label{eq.influence}
I_i(f)=\p\big(f(X)\neq f(\tau_iX)\big)
\end{equation}
where $\mathcal{L}(X)=\mu^n$ and $\tau_ix=(x_1,\ldots,-x_i,\ldots,x_n)$ for any $x\in C_n$ (i.e. $\tau_ix$ corresponds to the point $x$ with its $i$-th coordinate being flipped).
\end{defn}
\begin{rem}
\begin{enumerate}

\item For further purposes notice that $I_i(f)$ can also be equivalently expressed (if $f$ is a Boolean function) in terms of a $L^1(\mu^n)$ norm of some discrete derivative.  Namely, if the discrete derivative along the $i$-th coordinate is defined as
\[
D_i(f)=f(\tau_ix)-f(x)\quad \text{for any}\quad x\in C_n,
\]
we have $\|D_i(f)\|_1=I_i(f)$. In fact, for any $p\geq 1$,  $\|D_if\|_p^p=I_i(f)$ where $\|\cdot\|_p$ denote the norms of $L^p(\mu^n)$.
\end{enumerate}
\end{rem}

In \cite{BenLin}, the authors studied the influence of the coordinates of the the so-called Tribes function which is defined as follow : assume that $n=km$ and $x=(x_1,\ldots,x_{km})\in\{-1;1\}^{km}$, then 

\[
{\rm Tribes}_{km}(x)=\max_{i=1,\ldots, m}x^{(i)}
\]
where $x^{(i)}=\min\{x_{(i-1)k+1},\ldots, x_{ik}\}$ for any $i=1,\ldots,m$ . In particular, the function ${\rm Tribes}_{km}(x)$ takes the value $1$ if and only if, for some $i\in\{1,\ldots,m\}$, one of the tribes $(x_{(i-1)k+1},\ldots,x_{ik})$ of length $k$ is the tribes where all the coordinates are equal to $1$.\\

In their article, Ben-Or and Linial proved that the preceding function has all its coordinates with influence at least of order $\log n/n$. Besides, they have conjectured that this result is optimal. More precisely, we give below the statement of their result.

\begin{prop}[Ben-Or, Linial]\label{prop.benor.linial}
With the preceding notations, let $n$ be sufficiently large and set $k=\log n-\log\log n+\log\log 2$. Then, for all $i\in\{1,\ldots,n\}$, the following holds

\[
I_i({\rm Tribes}_n)=\frac{\log n}{n}\big(1+o(1)\big).
\]
\end{prop}

 Later on, in \cite{KKL}, Kahn, Kalai and Linial  have proved the conjecture. Namely 

\begin{thm}[Kahn-Kalai-Linial]\label{thm.kkl}
For any function $f\,:\,C_n\to\{0,1\}$ there exists $i~\in\{1,\ldots,n\}$  such that, for any $n\geq 1$,

\begin{equation}\label{eq.kkl}
I_i(f)\geq C{\rm Var}_{\mu^n}(f)\frac{\log n}{n}
\end{equation}
with ${\rm Var}_{\mu^n}(f)=\int_{C_n}f^2d\mu^n-\big(\int_{C_n}fd\mu^n\big)^2$ and $C>0$ is a numerical constant independent of $f$ and $n$.
\end{thm}

By convention, in the sequel, $C>0$ is a numerical constant that may change at each occurence.\\

As we will briefly explain below, Theorem \ref{thm.kkl} can be proved with the help of Talagrand inequality which can be stated as follows.

\begin{thm}[Talagrand]\label{thm.talagrand1}
For any function $f\,:\,C_n\,\to\R$, the following inequality holds

\begin{equation}\label{eq.talagrand1}
{\rm Var}_{\mu^n}(f)\leq C\sum_{i=1}^n\frac{\|D_i f\|_2^2}{1+\log\bigg(\frac{\|D_if\|_2}{\|D_if\|_1}\bigg)},
\end{equation}

\noindent where $C>0$ is an absolute numerical constant.
\end{thm} 

\begin{rem}
Talagrand inequality improves, by a logarithmic factor, upon the classical Poincar\'e inequality (up to numerical constant) :

\begin{equation}\label{eq.poincare}
{\rm Var}_{\mu^n}(f)\leq \frac{1}{4}\sum_{i=1}^n\|D_if\|_2^2.
\end{equation}

\noindent As mentioned before, \eqref{eq.talagrand1} can be used to provide an alternative proof of Theorem \ref{thm.kkl}. Indeed, consider $f\,:\,C_n\to\{0,1\}$ and recall that, for any $p\geq 1$, $\|D_if\|_p^p=I_i(f)$. Then, to deduce \eqref{eq.kkl} from \eqref{eq.talagrand1}, assume that $I_i(f)\leq \big(\frac{{\rm Var}_{\mu^n}(f)}{n}\big)^{1/2}$ for any $i\in\{1,\ldots,n\}$, since if not the results holds. Then, according to \eqref{eq.talagrand1}, there exists $i\in~\{1,\ldots,n\}$ such that 

\[
\frac{{\rm Var}_{\mu^n}(f)}{Cn}\leq \frac{I_i(f)}{1+\log \bigg(\frac{1}{\sqrt{ I_i(f)}}\bigg)}\leq \frac{4I_i(f)}{4+\log \bigg(\frac{n}{{\rm Var}_{\mu^n}(f)}\bigg)}
\]

\noindent which easily leads to \eqref{eq.kkl}.
\end{rem}

The aim of this note is to develop an interpolation method by semigroups together with hypercontractive arguments to reach Talagrand inequality at order two. That is to say : the new inequalities will be similar to \eqref{eq.talagrand1} with derivatives of order two instead. The following Theorem is the main result of this note.

\begin{thm}\label{thm.talagrand.ordre.superieur.cube.discret.ordre2}
Let $0<s_0<\frac{1}{128}$ be fixed. For any Boolean function $f\,:\,C_n\to~\{0,1\}$ and any $n\geq 1$, the following holds

\begin{eqnarray}\label{eq.talagrand.ordre.superieur.cube.discret.ordre2}
{\rm Var}_{\mu^n}(f)\leq C\bigg(\sum_{i=1}^n\|D_if\|_{1+e^{-2s_0}}^2+\sum_{\underset{i\neq j}{i,j=1}}^n\frac{\|D_{ij}f\|_2^2}{\bigg[1+\log\bigg(\frac{\|D_{ij}f\|_2}{\|D_{ij}f\|_1}\bigg)\bigg]^2}\bigg)
\end{eqnarray}

\noindent where $D_{ij}=D_i\circ D_j$ for any $i,j\in\{1,\ldots,n\}$ and $C>0$ is a numerical constant. 
\end{thm}
\begin{rem}
We want to highlight the fact that $s_0$ et $C$ are independent of $f$ and $n$.
\end{rem}

As an application of this result, we propose a theorem in the spirit of Theorem \ref{thm.kkl} with the influence $I_{(i,j)}(f)$ of a function $f$ for some coordinates $(i,j)\in\{1,\ldots,n\}^2$. This notion will be precisely defined in the sequel as an extension of the standard notion of influence \eqref{eq.influence}.

\begin{cor}\label{cor.kkl.ordre2}
Let $f\,:\, C_n\to\{0,1\}$ be a Boolean function. Then,  the following alternative holds : either there exists $i\in\{1,\ldots,n\}$ such that 

\[
I_i(f)\geq c\bigg({\rm Var}_{\mu^n}(f)\bigg)^{1/(1+\eta)}\bigg(\frac{1}{n}\bigg)^{1/(1+\eta)} \quad \text{with} \quad 0<\eta<1
\]

\noindent or there exists $(i,j)\in\{1,\ldots,n\}^2$ (with $i\neq j$) such that 

\[
I_{(i,j)}(f)\geq c{\rm Var}_{\mu^n}(f)\bigg(\frac{\log n}{n}\bigg)^2.
\]
In each case, $c>0$ and $\eta$ are absolute constants independent of $f$ and $n$.
\end{cor}

\noindent The rest of this paper is organized as follow : section \ref{two} provides semigroup tools and the framework of Boolean analysis needed to prove Theorem \ref{thm.talagrand.ordre.superieur.cube.discret.ordre2}.  Section \ref{three} is devoted to the proof of Theorem \ref{thm.talagrand.ordre.superieur.cube.discret.ordre2} and Corollary \ref{cor.kkl.ordre2} ; also, some remarks about extensions at higher orders will be given. In section \ref{four}, we present how can Theorem \ref{thm.talagrand.ordre.superieur.cube.discret.ordre2} extend in a Gaussian context. Finally, in the last section, we present some open questions related to our work and related to  some recent results in Concentration of Measure Theory (the so-called concentration at higher order for instance).

\section{Framework and tools}\label{two} 

\subsection{Some facts about semigroups}
The discrete cube $C_n=\{-1,1\}^n$ is an interesting example for which semigroups interpolation methods can be used to reach functional inequalities. Let us briefly collect some basic properties of this space equipped with the product measure $\mu^n$, where $\mu=\frac{1}{2}\delta_{-1}+\frac{1}{2}\delta_1$. \\

The classical semigroup associated to $(C_n,\mu^n)$ (cf. \cite{Odo,GarStei,CoLed}) is referred to the Bonami-Beckner semigroup $(Q_t)_{t\geq 0}$. As it is classical, $\mu^n$ is its invariant and reversible measure. Recall that the discrete Laplacian is given by  
\[
L=\frac{1}{2}\sum_{i=1}^n D_i
\]
with $D_i$ the (discrete) partial derivative along the $i$-th coordinate. This differential operator can be used to define a Dirichlet form on $C_n$ : for any functions $f,g\,:\,C_n\to\R$, we set

\begin{equation}\label{eq.ipp}
\mathcal{E}_{\mu^n}(f,g)=\int_{C_n}f(-Lg)d\mu^n=4\int_{C_n}\nabla f\cdot \nabla gd\mu^n,
\end{equation}

\noindent where $\nabla h=(D_1h,\ldots, D_nh)$ is the discrete gradient of any function $h\,:\,C_n\to\R$. The operator $L$ is also used to define the so-called Bonami-Beckner semigroup by the formula $Q_t=e^{tL}$ with $t\geq 0$. Now, let us recall some important properties of $(Q_t)_{t\geq 0}$.
 
 \begin{prop}
 \begin{enumerate}
\item The Bonami-Beckner semigroup admits an integral representation formula, for any $t\geq 0$, 
 
 \[
 Q_t(f)(x)=\int_{C_n}f(y)\prod_{i=1}^n(1+e^{-t}x_iy_i)d\mu^n(y)\quad \text{with}\quad x\in C_n.
 \]
 
\item  $(Q_t)_{t\geq 0}$ is Markovian and $\mu^n$ is its invariant and reversible measure. Namely, for any $t\geq 0$,  
 
 \[
Q_t(1)=1\quad \text{and}\quad  \int_{C_n}fQ_t(g)d\mu^n=\int_{C_n}gQ_t(f)d\mu^n
 \]
  for any functions $f,g\,:\,C_n\to\R$.   
  \end{enumerate}
  \end{prop}
  
  \begin{rem}
 With this integral representation in hand, it is easily seen that the following commutation formula holds
  \begin{equation}\label{eq.communtation.bonami}
  Q_tD_i=D_iQ_t
  \end{equation}
  for any $i\in\{1,\ldots,n\}$ and $t\geq 0$.
  \end{rem}
  
It has been proven (cf. \cite{Odo, Gro,DiaSal})  that $(Q_t)_{t\geq 0}$ satisfies an hypercontractive property. That is to say
 
\begin{thm}\label{thm.hypercontractivite.bonami.beckner}(Bonami-Beckner)
The semigroup $(Q_t)_{t\geq 0}$ is hypercontractive. Namely, for any $f\,:\,C_n\to\R$, every $t\geq 0$ and every $q\geq 1$

\begin{equation}\label{eq.hypercontractivite.bonami.beckner}
\|Q_t(f)\|_{q}\leq \|f\|_{p},
\end{equation}
\noindent with $p=p(t)=1+(q-1)e^{-2t}$.
\end{thm}

For further purposes, let us recall that the Poincar\'e inequality \eqref{eq.poincare} is equivalent to the following inequality, for any function $f \,: \, C_n \to \R$, we have

\[
 {\rm Var}_{\mu^n}\big(Q_t(f)\big)\leq e^{-2t} {\rm Var}_{\mu^n} (f) \quad \text{for any}\quad t\geq 0.
\]

\noindent Equivalently, when $f$ is centered under $\mu^n$, it reads

\begin{equation}\label{eq.exponential.poincare}
\|Q_tf\|_2^2\leq e^{-2t}\|f\|_2^2\quad \text{for any}\quad t\geq 0
\end{equation}

\noindent since $\mu^n$ is the invariant measure of $(Q_t)_{t\geq 0}$.\\

In particular, during the proof of our main result, inequality \eqref{eq.exponential.poincare} will be used with $D_if$ and $D_{ij}f$ for any $i,j=1,\ldots,n$. Indeed, notice that 

\begin{equation}\label{eq.derivative.center}
\int_{C_n}f(x)d\mu^n=\int_{C_n}f(\tau_ix)d\mu^n
\end{equation}

\noindent therefore, $D_if$ and $D_{ij}f$, for any $i,j=1,\ldots,n$, are centered under the measure $\mu^n$.

\subsection{Influences}
For more details, general references on Boolean Analysis are \cite{Odo,GarStei}. In this section we introduce the notion of the influence of a couple of coordinates $(i,j)\in\{1,\ldots,n\}^2$ which extends the classical notion of influence \eqref{eq.influence}.   \\
\newline

\begin{defn}
For any Boolean function $f\,:\,C_n\to\{0,1\}$ and for any $(i,j)\in~\{1,\ldots,n\}^2$, the influence of the couple $(i,j)$ of the function $f$ is given by

\begin{equation}\label{eq.influence.double}
I_{(i,j)}(f)=\frac{1}{2}\|D_{ij}f\|_1
\end{equation}

\noindent where $D_{ij}=D_i\circ D_j$. 
\end{defn}

\begin{rem}
\begin{enumerate}
\item It is easily seen that, for any $(i,j)\in\{1,\ldots,n\}^2$,  
\[
D_{ij}f=f(x)-f(\tau_ix)-f(\tau_jx)+f(\tau_{ij}x)\quad \text{for any} \quad x\in C_n
\]
 where $\tau_{ij}=\tau_i\circ\tau_j$. In particular, when $i=j$, $D_{ii}=2D_i$. Therefore \eqref{eq.influence.double} is an extension of the notion of influence (in its alternative formulation in terms of $L^1(\mu^n)$ norm).\\
\item As in the classical case, it is possible to show that, for any $i\neq j$,  $\|D_{ij}f\|_1$ and $\|D_{ij}f\|_2^2$ are equivalent. Indeed, for any  $(i,j)\in\{1,\ldots,n\}^2$ with $(i\neq j)$, we have

\begin{equation}\label{eq.comparaison}
\|D_{ij}f\|_1\leq \|D_{ij}f\|_2^2\leq 2\|D_{ij}f\|_1.
\end{equation}

From a heuristic point of view, this can be explained as follow : for any Boolean function $f\,:\,C_n\to\{0,1\}$ and any $(i,j)\in\{1,\ldots,n\}^2,\, i\neq j$ we have $|D_{ij}f|\in\{0,1,2\}$ for any $x\in C_n$. Besides, since $f$ is Boolean, there exists $A\subset C_n$ such that $f=1_A$. Then, it is enough to study, for $p\in\{1,2\}$,
\[
\int_{C_n}\big|f(x)-f(\tau_ix)-f(\tau_jx)+f(\tau_{ij}x)\big|^pd\mu^n(x)\quad \text{for}\quad x\in C_n
 \]
along the partition of $C_n$ induced by the set $A$. That is to say, it is enough to cut the integral according to the family of sets 
 \begin{eqnarray*}
 &\{x&\in C_n\,;\, x\in A,\, \tau_i(x)\notin A,\, \tau_j(x)\notin A, \tau_{ij}(x)\notin A\},\\
   &\{x&\in C_n\,;\, x\in A,\, \tau_i(x)\in A,\, \tau_j(x)\notin A,\, \tau_{ij}(x)\notin A\},\\
   &\{x&\in C_n\,;\, x\notin A,\, \tau_i(x)\notin A,\, \tau_j(x)\notin A,\, \tau_{ij}(x)\notin A\},\\
   &\{x&\in C_n\,;\, x\in A,\, \tau_i(x)\notin A,\, \tau_j(x)\in A,\, \tau_{ij}(x)\in A\},\\
   &\vdots&\\
   &\{x&\in C_n\,;\, x\in A,\, \tau_i(x)\in A,\, \tau_j(x)\in A,\, \tau_{ij}(x)\in A\},
 \end{eqnarray*}
to prove this fact.
\end{enumerate}
\end{rem}

\section{Proof of Theorem \ref{thm.talagrand.ordre.superieur.cube.discret.ordre2}}\label{three}

The proof starts with the representation of the variance of $f$ along the Bonami-Beckner's semigroup $(Q_t)_{t\geq 0}$ (cf. \cite{CoLed,BobGoHou}) :

\begin{equation}\label{eq.representation.variance.cube}
{\rm Var}_{\mu^n}(f)=2\int_0^\infty \sum_{i=1}^n\int_{C_n}Q_t^2(D_if)d\mu^ndt.
\end{equation}

\noindent Then, set $2s=t$ and for any $i=1,\ldots,n$, use the fact that 

\[
\int_{C_n}Q_{2s}^2(D_if)d\mu^n=\|Q_s\circ Q_s(D_if)\|_2^2\leq e^{-2s}\|Q_s (D_if)\|_2^2
\]
where the last upper bound comes from the exponential decay in $L^2(\mu^n)$ of the semigroup \eqref{eq.exponential.poincare}. This gives the following upper bound, 

\begin{equation}\label{eq.proof0}
{\rm Var}_{\mu^n}(f)\leq 4\int_0^\infty e^{-2s} \sum_{i=1}^n\int_{C_n}Q_s^2(D_if)d\mu^n ds.
\end{equation}

\noindent Then, set 
\[
K(s)=\sum_{i=1}^n\int_{C_n}Q_s^2(D_if)d\mu^n\quad \text{for any}\quad s\geq 0.
\]
By a further integration by parts \eqref{eq.ipp} and applying again the  fundamental theorem of calculus, we get for any $s\geq 0$
\[
K(s)=K(\infty)-\int_s^\infty K'(u)du=K(\infty)+2\sum_{i,j=1}^n\int_s^\infty \int_{C_n}Q^2_u(D_{ij}f)d\mu^n du.
\]
Besides, by ergodicity, we have

\[
K(\infty)=\sum_{i=1}^n\bigg(\int_{C_n}D_ifd\mu^n\bigg)^2=0,
\]
where the last equality comes from the fact that, for any $i\in\{1,\ldots,n\}$,  $D_i f$ is centered under the measure $\mu^n$. Therefore, we have

\begin{equation}\label{eq.proof1}
K(s)=2\sum_{i,j=1}^n\int_s^\infty\int_{C_n}Q_u^2(D_{ij}f)d\mu^ndu\quad \text{for any}\,\, s\geq 0. 
\end{equation}

\noindent Substitute \eqref{eq.proof1} into \eqref{eq.proof0} and apply Fubini's theorem to get

\[
{\rm Var}_{\mu^n}(f)\leq 4\sum_{i,j=1}^n\int_0^\infty (1-e^{-2u})\int_{C_n}Q_u^2(D_{ij}f)d\mu^ndu.
\]

\noindent Again, set $2s=u$ and use the exponential decay of $(Q_t)_{t\geq 0}$ in $L^2(\mu^n)$ :

\[
\text{i.e.}\quad \|Q_{2s}(D_{ij}f)\|_2^2\leq e^{-2s}\|Q_s(D_{ij}f)\|_2^2\quad \text{for any}\quad  (i,j)\in\{1,\ldots,n\}^2.
\]
This yields

\[
{\rm Var}_{\mu^n}(f)\leq 8\sum_{i,j=1}^n\int_0^\infty e^{-2s}(1-e^{-4s})\int_{C_n}Q_s^2(D_{ij}f)d\mu^nds.
\]

\noindent Now, cut the sum in two parts (if $i=j$ or not). Notice that $D_{ii}=2D_i$ for any $i=1,\ldots,n$. The variance of $f$ is now bounded by two terms
\[
32\sum_{i=1}^n\int_0^\infty e^{-2s}(1-e^{-4s})\int_{C_n}Q_s^2(D_{i}f)d\mu^nds+8\sum_{i\neq j}\int_0^\infty e^{-2s}(1-e^{-4s})\int_{C_n}Q_s^2(D_{ij}f)d\mu^nds.
\]

\noindent  Let $s_0>0$ be a parameter to be chosen later. The first term of the preceding sum is managed as follow

\begin{eqnarray*}
32\sum_{i=1}^n\int_0^\infty e^{-2s}(1-e^{-4s})\int_{C_n}Q_s^2(D_{i}f)d\mu^nds&=&32\sum_{i=1}^n\int_0^{s_0} e^{-2s}(1-e^{-4s})\int_{C_n}Q_s^2(D_{i}f)d\mu^nds\\
&+&32\sum_{i=1}^n\int_{s_0}^\infty e^{-2s}(1-e^{-4s})\int_{C_n}Q_s^2(D_{i}f)d\mu^nds
\end{eqnarray*}

\noindent It is obvious to see that, 

\begin{eqnarray*}
\sum_{i=1}^n\int_0^{s_0} e^{-2s}(1-e^{-4s})\int_{C_n}Q_s^2(D_{i}f)d\mu^nds &\leq &\sum_{i=1}^n\int_0^{s_0} 4s\int_{C_n}Q_s^2(D_{i}f)d\mu^nds\\
&\leq& 4s_0\int_0^\infty \sum_{i=1}^n\int_{C_n}Q_s^2(D_if)d\mu^nds\\
&=&2 s_0 {\rm Var}_{\mu^n}(f)
\end{eqnarray*}
\noindent where the last equality comes from the dynamical representation of the variance along the semigroup \eqref{eq.representation.variance.cube}. Therefore, we have 

\begin{eqnarray*}
{\rm Var}_{\mu^n}(f)&\leq&32\sum_{i=1}^n\int_{s_0}^\infty e^{-2s}(1-e^{-4s})\int_{C_n}Q_s^2(D_{i}f)d\mu^nds\\
&+& 8\sum_{i\neq j}\int_0^\infty e^{-2s}(1-e^{-4s})\int_{C_n}Q_s^2(D_{ij}f)d\mu^nds +64s_0{\rm Var}_{\mu^n}(f)
\end{eqnarray*}

\noindent Now, let us choose $s_0$ such $64s_0\leq 1/2$ ; it yields 

\begin{eqnarray*}
\frac{1}{2}{\rm Var}_{\mu^n}(f)&\leq& 32\int_{s_0}^\infty e^{-2s}(1-e^{-4s})\sum_{i=1}^n\int_{C_n}Q_s^2(D_if)d\mu^nds\\
&+&8\sum_{i\neq j}\int_0^\infty e^{-2s}(1-e^{-4s})\int_{C_n}Q_s^2(D_{ij}f)d\mu^nds.
\end{eqnarray*}

The hypercontractive property \eqref{eq.hypercontractivite.bonami.beckner} of the Bonami-Beckner's semigroup can be used to bound the first integral (of the right hand side) of the preceding inequality. Since by Jensen inequality, $(Q_t)_{t\geq 0}$ is also a contraction of $L^2(\mu^n)$, for any $i=1,\ldots,n$ and every $s\geq s_0$, we also have

\[
\|Q_s(D_if)\|_2^2=\|Q_{s-s_0}\circ Q_{s_0}(D_if)\|_2^2\leq \|Q_{s_0}(D_if)\|_2^2.
\]

\noindent Thus, 

\begin{eqnarray*}
32\int_{s_0}^\infty e^{-2s}(1-e^{-4s})\sum_{i=1}^n\int_{C_n}Q_s^2(D_if)d\mu^nds&\leq&32\sum_{i=1}^n\|Q_{s_0}(D_if)\|_2^2\int_{s_0}^\infty e^{-2s}(1-e^{-4s})ds\\
&\leq&16 \sum_{i=1}^n\|Q_{s_0}(D_if)\|_2^2\\
&\leq& 16\sum_{i=1}^n\|D_if\|_{1+e^{-2s_0}}^2
\end{eqnarray*}

\noindent where, in the last inequality, we used the hypercontractive property \eqref{eq.hypercontractivite.bonami.beckner}. To conclude the proof, we have to bound the sum when $i\neq j$.

\begin{eqnarray*}
I&=&8\sum_{i\neq j}\int_0^\infty e^{-2s}(1-e^{-4s})\int_{C_n}Q_s^2(D_{ij}f)d\mu^nds\\
&\leq& 16\sum_{i\neq j}\int_0^\infty e^{-2s}(1-e^{-2s})\int_{C_n}Q_s^2(D_{ij}f)d\mu^nds.
\end{eqnarray*}

Again, by the hypercontractive property \eqref{eq.hypercontractivite.bonami.beckner} of $(Q_t)_{t\geq 0}$ we have, for any function $g\,:C_n\,\to\R$,

\[
\|Q_t(g)\|_2^2\leq \|g\|_{1+e^{-2t}}^2\quad \text{for any}\quad  t\geq 0.
\]

\noindent Apply this to $g=D_{ij}f$, for any $i,j=1,\ldots,n$ with $i\neq j$. Then, set $v=1+e^{-2t}$ to get

\begin{equation}\label{eq.hypercontractive.estimates}
I\leq 16\sum_{i\neq j}\int_1^2 (2-v)\|D_{ij}f\|_v^2dv.
\end{equation}

\noindent Furthermore, H\"older's inequality yields $\|D_{ij}f\|_v\leq \|D_{ij}f\|_1^\theta\|D_{ij}f\|_2^{1-\theta}$, with $\theta=\theta(v)$ satisfying $\frac{1}{v}=\frac{\theta}{1}+\frac{1-\theta}{2}$, for any $v\in[1,2]$. To sum up, we have

\[
I\leq16\sum_{i\neq j}\|D_{ij}f\|_2^2 \int_1^2(2-v)\bigg(\frac{\|D_{ij}f\|_1}{\|D_{ij}f\|_2}\bigg)^{2\theta}dv.
\]

\noindent Now, set $\alpha=\frac{\|D_{ij}f\|_1}{\|D_{ij}f\|_2}\leq 1$, after a change of variables, we easily obtain, for $i\neq j$, 

\[
\int_1^2(2-v)\bigg(\frac{\|D_{ij}f\|_1}{\|D_{ij}f\|_2}\bigg)^{2\theta}dv=\int_0^1ue^{-\frac{2u}{2-u}\log (1/\alpha)}du.
\]

\noindent Then, observe that $\int_0^1ue^{-\frac{2u}{2-u}\log (1/\alpha)}du\leq \frac{C}{\big[1+\log(1/\alpha)\big]^2}$ with $C>0$ a numerical constant. Finally, we have

\[
I\leq C\sum_{i\neq j}^n\frac{\|D_{ij}f\|_2^2}{\bigg[1+\log\bigg(\frac{\|D_{ij}f\|_2}{\|D_{ij}f\|_1}\bigg)\bigg]^2}
\]

\begin{rem}
The scheme of proof can be extended to higher order with minor modifications. For instance, for the order three, cut the sum in three parts : 
\begin{itemize}
\item [$\bullet$] the diagonal terms will give derivatives $D_i$ of order one ;\\
\item  [$\bullet$] when two indexes are equal we will obtain derivatives $D_{ik}$ of order two ;\\
\item [$\bullet$] the other terms will give derivatives $D_{ijk}=D_i\circ D_j\circ D_k$ of order three.\\
\end{itemize} 
Then, it is possible to apply the same methodology. Since the notations are a little bit heavy, we leave the details to the reader.
\end{rem}

\subsection{Proof of Corollary \ref{cor.kkl.ordre2}}

With the Theorem \ref{thm.talagrand.ordre.superieur.cube.discret.ordre2} at hand we can prove Corollary \ref{cor.kkl.ordre2}. \\

Consider $f\,:\,C_n\to\{0,1\}$ a Boolean function. Then, apply inequality \eqref{eq.talagrand.ordre.superieur.cube.discret.ordre2} from Theorem \ref{thm.talagrand.ordre.superieur.cube.discret.ordre2} to $f$ .\\

Then, thanks to the formulation of influences in terms of $L^p$ norms of  partial derivatives, observe that $\|D_if\|_{1+e^{-2s_0}}^2=\big[I_i(f)\big]^{2/(1+e^{-2s_0})}$ for any $i=1,\ldots,n$. Besides, for any $s_0\geq 0$, notice that $\frac{2}{1+e^{-2s_0}}\in(1,2)$. Therefore, with $s_0\in[0,\frac{1}{128}]$ being fixed, this can be rewritten as $1+\eta$ with $0<\eta<1$ where $\eta=\eta(s_0)$ is independent of $f$ and $n$.\\
\newline

\noindent Now, recall \eqref{eq.comparaison} which gives, for any $i\neq j$,

\[
\|D_{ij}f\|_1\leq \|D_{ij}f\|_2^2\leq 2\|D_{ij}f\|_1.
\]

\noindent Thus, since $I_{(i,j)}(f)=\frac{1}{2}\|D_{ij}f\|_1$ for any $i\neq j$, 

\[
{\rm Var}_{\mu^n}(f)\leq C\sum_{i=1}^nI_i(f)^{1+\eta}+C\sum_{i\neq j}\frac{I_{(i,j)}(f)}{\bigg[1+\log\bigg(\frac{1}{\sqrt{4I_{(i,j)}(f)}}\bigg)\bigg]^2}.
\]

\noindent If the first sum is larger than the second one, we get

\[
{\rm Var}_{\mu^n}(f)\leq C\sum_{i=1}^nI_i(f)^{1+\eta}.
\]

\noindent Thus, there exists some $i\in\{1,\ldots,n\}$ such that $I_i(f)^{1+\eta}\geq \frac{{\rm Var}_{\mu^n}(f)}{ Cn}$. If it is not the case, we obtain 

\[
{\rm Var}_{\mu^n}(f)\leq C\sum_{i\neq j}\frac{I_{(i,j)}(f)}{\bigg[1+\log\bigg(\frac{1}{\sqrt{4I_{(i,j)}(f)}}\bigg)\bigg]^2}.
\]

\noindent To conclude, it is enough to follow the scheme of proof presented in the introduction (below the equation \eqref{eq.poincare}). We leave the details to the reader.

\begin{rem}
\begin{enumerate}
\item Following (with minor and obvious variations) the proof of Proposition \ref{prop.benor.linial}, it is possible to show  that the influences $I_{(i,j)}({\rm Tribes}_{km})$, for $i\neq j$, are precisely of order $\frac{\log^2n}{n^2}$.\\
\item As communicated to us by Krzysztof Oleszkiewicz (cf. \cite{Oles2})  an alternative argument based on spectral decomposition and logarithmic Sobolev inequality can be used to reach conclusion which is similar to the one obtained in Corollary \ref{cor.kkl.ordre2}.\\
\end{enumerate}
\end{rem}

\section{Extension to a Gaussian setting}\label{four}

It is well known (cf. \cite{Chatt1}) that Talagrand inequality has also been obtained for the standard Gaussian measure $\gamma_n$ on $\R^n$. We want to emphasize the fact that the interpolation method (with the exact same arguments) used for $(C_n,\mu^n)$ also work $(\R^n,\gamma_n)$ with the Ornstein-Uhlenbeck semigroup instead.\\

First, we will briefly remind the reader of some properties of such semigroup (for more details we refer the reader to \cite{BGL}). Then, we present a variance representation formula, which already appeared under a different form in \cite{HPAS}. This representation formula can be seen as a Taylor expansion of the variance of $f$ with some remainder term. Finally we will briefly explain how the proof can be done with the help of the arguments used during the proof of Theorem \ref{thm.talagrand.ordre.superieur.cube.discret.ordre2}.

\subsection{Ornstein-Uhlenbeck semigroup}
This section gather some essential properties of the Ornstein-Uhlenbeck semigroup $(P_t)_{t\geq 0}$. Let $f\,:\,\R^n\to\R$ be a smooth function, the Ornstein-Uhlenbeck's semigroup satisfies the following properties.

 \begin{prop} 
 $(P_t)_{t\geq 0}$ is Markovian and $\gamma_n$ is its invariant and reversible measure. Namely, for any $t\geq 0$, we have
 
 \[
 P_t(1)=1\quad \text{and} \quad\int_{\R^n}fP_t(g)d\gamma_n=\int_{\R^n}gP_t(f)d\gamma_n,
 \]
 
 \noindent for any smooth functions $f,g\,:\,\R^n\to\R$. The Ornstein-Uhlenbeck's semigroup admits a integral representation formula,

 
 \[
 P_t(f)(x)=\int_{\R^n}f\big(e^{-t}x+\sqrt{1-e^{-2t}}y\big)d\gamma_n(y)
  \]
  for any $x\in\R^n$ and any $t\geq 0$.
\end{prop}
  
  \begin{rem}
 This integral representation easily leads to the following commutation property  between the semigroup and the gradient $\nabla$ :
 
 \begin{equation}\label{eq.commutation.ou}
\nabla P_t =e^{-t} P_t\nabla\quad \text{for any}\quad t\geq 0.
 \end{equation}
 
 \noindent An integration by parts formula also holds in this setting. Indeed, denote by $L=\Delta-x\cdot\nabla$ the infinitesimal generator of $(P_t)_{t\geq 0}$, then for any  smooth functions $f\,:\,\R^n\to\R$ and $g\,:\,\R^n\to\R$ it holds

 \begin{equation}\label{eq.ipp.ou}
 \int_{\R^n}f(-Lg)d\gamma_n=\int_{\R^n}\nabla f \cdot\nabla g d\gamma_n.
 \end{equation}
\end{rem}
  
It has been proven (cf. \cite{BGL, Nel}) that $(P_t)_{t\geq 0}$ also satisfies an hypercontractive property.
 
\begin{thm}[Nelson]\label{thm.hypercontracvite.ou}
The semigroup $(P_t)_{t\geq 0}$ is hypercontractive. Namely, for any $f\,:\,\R^n\to\R$ smooth enough, every $t\geq 0$ and every $p\geq 1$

\begin{equation}\label{eq.hypercontractivite.ou}
\|P_t(f)\|_{q}\leq \|f\|_{p},
\end{equation}
\noindent with $p=p(t)=1+(q-1)e^{-2t}$.
\end{thm}

\subsection{Variance representation}

\noindent The theorem below will be crucial to reach the version of Theorem \ref{thm.talagrand.ordre.superieur.cube.discret.ordre2} in a Gaussian setting.\\

 In the sequel, $\nabla^2f$ will stand for the Hessian matrix of any smooth function $f\,:\,\R^n\to\R$ and, with obvious notations, $\nabla^pf$ (with $p\geq 2$) corresponds to higher order. We said that $f\in\mathcal{C}^m(\R^n)$ if , for every $ {\displaystyle \alpha _{1},\alpha _{2},\ldots ,\alpha _{n}}$ non-negative integers, such that ${\displaystyle \alpha =\alpha _{1}+\alpha _{2}+\cdots +\alpha _{n}\leq m}$, 
 \[
{\displaystyle {\frac {\partial ^{\alpha }f}{\partial x_{1}^{\alpha _{1}}\,\partial x_{2}^{\alpha _{2}}\,\cdots \,\partial x_{n}^{\alpha _{n}}}}}
\]
 exists and is continuous on $\R^n$.
 
\begin{thm}\label{thm.developpement.variance.taylor}
Within the preceding framework, consider $f\,:\,\R^n\to\R$ and assume that there exists $m\geq 1$ such that $f\in\mathcal{C}^m(\R^n)$. Assume also that $f$ and all its partial derivatives belong to $L^2(\gamma_n)$. Then, for every $1\leq p\leq m-1$, we have the following representation formula 

\begin{equation}\label{eq.developpement.variance.taylor}
{\rm Var}_{\gamma_n}(f)=\sum_{k=1}^p\frac{1}{k!}\bigg|\int_{\R^n}\nabla^kfd\gamma_n\bigg|^2+\frac{2}{p!}\int_0^\infty e^{-2t}\big(1-e^{-2t}\big)^p \int_{\R^n}\big|P_t(\nabla^{p+1}f)\big|^2d\gamma_ndt, 
\end{equation}
where $|\cdot|$ stands for the Euclidean norm.

\end{thm}

\begin{rem}
\begin{enumerate}
\item Notice that, when $p\to\infty$, the formula \eqref{eq.developpement.variance.taylor} yields, up to integration by parts, the decomposition of a function of $L^2(\gamma_n)$ along the Hermite polynomial basis (cf. \cite{BGL}).\\
\item In his article \cite{Led4}, Ledoux uses similar interpolation arguments (with the interval $[0,t]$ instead of $[t,+\infty[$) in order to obtain another representation formula for the variance of a function $f$. \\
\item  As in \cite{Led4},  the same proof can be performed with the entropy instead of the variance. However, formulas are not so easily handled. For instance, at the first iteration of the method we obtain, for $f\,\,:\,\R^n\to\R$ such that $f>0$, 

\[
2{\rm Ent}_{\gamma_n}(f)=\frac{\bigg|\int_{\R^n}\nabla fd\gamma_n\bigg|^2}{\int_{\R^n}fd\gamma_n}+\int_0^\infty e^{-2u}(1-e^{-2u})\int_{\R^n}k_ud\gamma_n du
\]
with $k_u=(P_uf)^{-3}\big|P_u(\nabla f)^{t}P_u(\nabla f)-P_u(f)P_u(\nabla^2f)\big|^2$ and where

 \[
 {\rm Ent}_{\gamma_n}(f)=\int_{\R^n}f\log fd\gamma_n-\bigg(\int_{\R^n}fd\gamma_n\bigg)\bigg(\log \int_{\R^n}fd\gamma_n\bigg).
 \]
Since $k_u\geq 0$ for every $u\geq0$, this implies, for any $f$ such that $\int_{\R^n}fd\gamma_n$=1,

\[
2{\rm Ent}_{\gamma_n}(f^2)\geq \bigg|\int_{\R^n}\nabla fd\gamma_n\bigg|^2.
\]
\noindent This lower bound corresponds to the inverse logarithmic Sobolev inequality (cf. \cite{BGL}).\\
\end{enumerate}
\end{rem}

\begin{proof}(of Theorem \ref{thm.developpement.variance.taylor})\\
The starting point of the proof is the dynamical representation of the variance of a function $f\,:\R^n\to\R$, along the Ornstein-Uhlenbeck's semigroup (cf. \cite{BGL})

\[
{\rm Var}_{\gamma_n}(f)=2\int_0^\infty e^{-2t}\int_{\R^n}\big|P_t(\nabla f)\big|^2d\gamma_n dt.
\]

\noindent Set 
\[
K_1(t)=\int_{\R^n}\big|P_t(\nabla f)|^2d\gamma_n\quad \text{for any}\quad t\geq 0.
\]

\noindent Then, according the fundamental theorem of calculus, for any $0\leq t\leq s$, we have
\[
K_1(t)=K_1(s)-\int_t^sK'_1(u)du
\]

\noindent using the fact that $\nabla P_u (f)=e^{-u}P_u(\nabla f)$ and the integration by parts formula \eqref{eq.ipp.ou}, we obtain
\[
K'_{1}(u)=\frac{d}{du}\int_{\R^n}|P_u (\nabla f)|^2d\gamma_n=-2\int_{\R^n}e^{-2u}|P_u(\nabla^2f)|^2d\gamma_n
\]

\noindent Finally, for every $0\leq t\leq s$, 
\[
K_1(t)=K_1(s)+2\int_t^se^{-2u}\int_{\R^n}|P_u(\nabla^2f)|^2d\gamma_ndu,
\]

\noindent Thus, when $s\to\infty$, 

\[
K_1(t)=\bigg|\int_{\R^n}\nabla fd\gamma_n\bigg|^2+2\int_t^\infty e^{-2u}\int_{\R^n}|P_u(\nabla^2f)|^2d\gamma_ndu,
\]

\noindent by ergodicity of $(P_t)_{t\geq 0}$. Substitute $K_1$ into the representation formula to get

\[
{\rm Var}_{\gamma_n}(f)=\bigg|\int_{\R^n}\nabla fd\gamma_n\bigg|^2+4\int_0^\infty e^{-2t}\int_t^\infty e^{-2u}\int_{\R^n}|P_u(\nabla^2 f)|^2d\gamma_ndudt.
\]

\noindent Then, by Fubini's Theorem, 

\[
{\rm Var}_{\gamma_n}(f)=\bigg|\int_{\R^n}\nabla fd\gamma_n\bigg|^2+2\int_0^\infty e^{-2u}(1-e^{-2u})\int_{\R^n}|P_u(\nabla^2 f)|^2d\gamma_ndu.
\]

\noindent In order to obtain the general statement, iterate the scheme of proof : set similarly

\[
K_2(u)=\int_{\R^n}|P_u(\nabla^2 f)|^2d\gamma_n,
\]

\noindent then

\[
K_2(u)=\bigg|\int_{\R^n}\nabla^2fd\gamma_n\bigg|^2+2\int_u^\infty e^{-2t}\int_{\R^n}|P_t(\nabla^3f)|^2d\gamma_ndt.
\]

\noindent After some substitution, it is enough to calculate

\[
\bigg|\int_{\R^n}\nabla^2fd\gamma_n\bigg|^2\times\bigg[2\int_0^\infty e^{-2u}(1-e^{-2u})du\bigg]=\frac{1}{2}\bigg|\int_{\R^n}\nabla^2fd\gamma_n\bigg|^2
\]

\noindent and

\[
4\int_0^\infty e^{-2t}\bigg(\int_{\R^n}|P_t (\nabla^3f)|^2d\gamma_n\bigg)\times\bigg[\int_0^te^{-2u}(1-e^{-2u})du\bigg]dt.
\]

\noindent A straightforward calculus yields
 
 \[
 2\int_0^te^{-2u}(1-e^{-2u})du=\big(1-e^{-2t}\big)^2\quad \text{for any}\quad t\geq 0.
 \]
 
\noindent Then, proceed by induction to conclude. Indeed, we can define by induction the coefficients that appeared at each iteration. To this task, set 
\[
a_0(t)=2e^{-2t}\quad \text{for any}\quad t\geq 0\quad  \text{and}\quad  a_1=\int_0^\infty a_0(t)dt.
\]
Then, for $k\geq 1$, $a_k(t)=a_0(t)\int_0^ta_{k-1}(u)du$ and $a_k=\int_0^\infty a_k(t)dt$. It is not difficult to show that, for every $k\geq 0$ and every $t\geq 0$, 

\[
a_k(t)=\frac{2}{k!}e^{-2t}\big(1-e^{-2t}\big)^k.
\]

\noindent Thus, for every $k\geq 0$, $a_k=\frac{1}{k!}$.
\end{proof}

\subsection{Taylor expansion of the variance with remainder term}
We focus on the particular case $p=1$. We present below what can be deduced from the representation formula \eqref{eq.developpement.variance.taylor}.

 \subsubsection{Order 1}
For $p=1$, the representation formula of the variance tells us that

\begin{equation}\label{eq.representation2}
{\rm Var}_{\gamma_n}(f)=\bigg|\int_{\R^n}\nabla fd\gamma_n\bigg|^2+2\int_0^\infty e^{-2t}(1-e^{-2t})\int_{\R^n}|P_t(\nabla^2f)|^2d\gamma_ndt.
\end{equation}

The second term is always strictly positive, so it implies an inverse Poincar\'e inequality (cf. \cite{BGL})

\[
{\rm Var}_{\gamma_n}(f)\geq \bigg|\int_{\R^n}\nabla fd\gamma_n\bigg|^2.
\]

\noindent It is also possible to control the remainder term in order to upper bound the variance of $f$. Indeed, based on \eqref{eq.representation2} , we can apply the hypercontractive scheme of proof (of Talagrand inequality) to reach the following Theorem.

\begin{thm}\label{thm.talagrand.ordre2.gaussien}
Within the preceding framework, for any function $f\,:\,\R^n\to\R$ smooth enough, we have 
\[
{\rm Var}_{\gamma_n}(f)\leq\bigg|\int_{\R^n}\nabla fd\gamma_n\bigg|^2+C\sum_{i,j=1}^n\frac{\|\partial^2_{ij}f\|_2^2}{\bigg[1+\log \frac{\|\partial^2_{ij}f\|_2}{\|\partial^2_{ij}f\|_1}\bigg]^2},
\]
with $C>0$ a universal numerical constant.


\end{thm}

\begin{proof}
Start with the representation formula \eqref{eq.representation2}

\[
2\int_0^\infty e^{-2t}(1-e^{-2t})\int_{\R^n}|P_t(\nabla^2f)|^2d\gamma_ndt=2\sum_{i,j=1}^n\int_0^\infty e^{-2t}(1-e^{-2t})\int_{\R^n}\big(P_t(\partial^2_{ij}f)\big)^2d\gamma_ndt,
\]

\noindent Then, it is enough to bound 

\[
I=2\sum_{i,j=1}^n\int_0^\infty e^{-2t}(1-e^{-2t})\|P_t(\partial_{ij}f)\|_2^2dt.
\]

\noindent To this task, use the hypercontractive property \eqref{thm.hypercontracvite.ou} of  $(P_t)_{t\geq 0}$. Namely,  for any function smooth $g\,:\R^n\,\to\R$, 

\[
\|P_t(g)\|_2^2\leq \|g\|_{1+e^{-2t}}^2,\quad t\geq 0.
\]

\noindent with $g=\partial_{ij}f$, for any $i,j=1,\ldots,n$ and follow the exact same estimates that has been used after inequality \eqref{eq.hypercontractive.estimates}. We leave the details to the reader.
\end{proof}

\noindent Similarly, as the discrete case, it is possible to extend Theorem \ref{thm.talagrand.ordre2.gaussien} at higher order. Notice again that the second term (with the logarithmic factor) can be seen as the remainder term of Taylor's expansion of the variance. 
\begin{thm}
Let  $f\,:\R^n\to\R$ be such that $f\in C^p(\R^n)$ for some $p\geq 1$ and all its partial derivatives (up to order $p$) belong to the space $L^2(\gamma_n)$. Then, for any $p\geq 1$, we have

\[
{\rm Var}_{\gamma_n}(f)\leq \sum_{k=1}^p\frac{1}{k!}\bigg|\int_{\R^n}\nabla^kfd\gamma_n\bigg|^2+C\sum_{i_1,\ldots,i_{p+1}=1}^n\frac{\|\partial_{i_1,\ldots,i_{p+1}}f\|_2^2}{\bigg[1+\log\bigg(\frac{\|\partial_{i_1,\ldots,i_{p+1}}f\|_2}{\|\partial_{i_1,\ldots,i_{p+1}}f\|_1}\bigg)\bigg]^{p+1}}
\]
with $C>0$ a numerical constant.






\end{thm}
\begin{rem}
Observe that the sum $\sum_{k=1}^p\frac{1}{k!}\bigg|\int_{\R^n}\nabla^kfd\gamma_n\bigg|^2$ is precisely the beginning of the expansion of a function $f\in L^2(\gamma_n)$ along the Hermite's polynomials basis.
\end{rem}

\section{Further comments and remarks}\label{five}

To conclude this note, we would like to make some remarks about the potential extension of our work.

\subsection{Potential extensions}
Let us start with the discrete cube.
\subsubsection{Biased cube}
It is possible to equip the discrete cube $\{-1,1\}^n$ with a biased measure  $\nu_p^n=(p\delta_1+q\delta_{-1})^{\otimes n}$ with $p\in[0,1]$ and $q+p=1$. This measure also satisfied a Poincar\'e and logarithmic Sobolev inequalities (cf. \cite{Odo, CoLed}). $\nu_p^n$ is also the invariant measure of  an hypercontractive and ergodic semigroup $(T_t^p)_{t\geq 0}$. It is then obvious that our results can be immediately extended to such setting. However, some care has to be taken with the constant involved in the proof : some of them will depend on the logarithmic Sobolev constant $pq\frac{\log p-\log q}{p-q},\,p\neq q$ of $\nu_p^n$.\\

The study of the dependence in $p$ of the measure $\nu_p^n$ has been proven useful (cf. \cite{BLM, TalL1L2} for more details) concerning sharp threshold for monotone graph. For instance, in \cite{FrieKal}, the authors proved the following

\begin{thm}[Friedgut-Kalai]\label{thm.friedgut.kalai}
For every symmetric monotone set $A$ and every $\epsilon>0$, if $\nu_p^n(A)>\epsilon$ then $\nu_q^n(A)>1-\epsilon$ for $q=p+c_1\log(1/2\epsilon)/\log n$ where $c_1$ is an absolute constant.
\end{thm}

They also asked if the following holds (cf. \cite{FrieKal} for more details)

\begin{conjec}\label{conjec.friedgut.kalai}
Let $P$ be any monotone property of graphs on $n$ vertices and $\epsilon>0$. If
$\nu_p^n(P)>\epsilon$, then $\nu_q^n(P)>1-\epsilon$ for $q=p+c\log(1/2\epsilon)/\log^2 n$.
\end{conjec}

The proof fo Theorem \ref{thm.friedgut.kalai} relies on the so-called Russo-Margulis's Lemma (cf. \cite{BLM, FrieKal}) and Kahn-Kalai-Linial Theorem \ref{thm.kkl}. It is then natural to ask if Talagrand inequalities at order two (and its consequences in terms of influences) for the biased cube can be used to prove Conjecture \ref{conjec.friedgut.kalai} ?\\

As a matter of fact, it can be shown (with elementary calculus) that Russo-Margulis's Lemma can be extended at order two. However it seems (cf. \cite{Ross}) that the extension of Kahn-Kalai-Linial's theorem at order two is too rough to prove the conjecture. Maybe one should add further arguments.

\subsubsection{General setting}
As another extension of our work, it is possible to consider the general framework of Cordero-Erausquin and Ledoux's article \cite{CoLed}. Indeed, as they have investigated in their paper, the crucial point of Talagrand inequality \eqref{eq.talagrand1} is the decomposition of the Dirichlet energy along directions which commutes with the semigroup (cf. \cite{CoLed} for more details) together with some hypercontractive estimates. Even if this extension is straightforward, we did not want to get into this level of generality for the sake of clarity of our exposition. However, we want to emphasize that $(C_n, \nu_p^n)$ and $(\R^n, \gamma_n)$ (and more general measures) fit this setting.

\subsection{Links with concentration of measure}

As far as we are concerned, it seems that our work has some connection with some recent results of Concentration of Measure Theory. General references for this topic are \cite{Led, BLM}.\\

\noindent In a Gaussian setting, the Concentration of Measure phenomenon is usually stated as follow.

\begin{thm}[Borell-Sudakov-Tsirel'son-Ibragimov]\label{thm.borel}
Let $f\,:\,\R^n\to \R$ be a Lipschitz function and $X$ a standard Gaussian vector in $\R^n$. Then, the following holds

\begin{equation}\label{eq.borell}
\p\bigg( \big|f(X)-\E[f(X)]|\big|\geq t\bigg)\leq 2e^{-t^2/2\|f\|_{Lip}^2}\quad \text{for any}\quad t\geq 0.
\end{equation}
where $\|f\|_{Lip}=\sup\bigg\{\frac{|f(x)-f(y)|}{|x-y|},\, x,y\in\R^n,\, x\neq y\bigg\}$.
 \end{thm}

This result is known to be sharp for the large deviation regime (cf. \cite{LT,PaouVal}). Nevertheless, it is not the case for the small deviation regime as it can been seen on the Lipschitz function $f(x)=\max_{i=1,\ldots,n}x_i$.\\

\subsubsection{Superconcentration inequalities}
 In their article \cite{PaouVal2}, Paouris and Valettas, proved that Talagrand inequality (in a Gaussian setting) can be used to precise inequality \eqref{eq.borell} in the small deviation regime. More precisely, they proved the following

\begin{prop}[Paouris-Valettas]\label{prop.paou.valettas}
Let $f\,:\,\R^n\to\R$ be a Lipschitz function with

\[
|f(x)-f(y)|\leq b\|x-y\|_2,\quad |f(x)-f(y)|\leq a\|x-y\|_{\infty},\quad x,y\in\R^n
\]

\noindent and $\|\partial_i f\|_1\leq A$ for all $i\in\{1,\ldots,n\}$. Then, if we set $F=f-\E_{\gamma_n}[f]$,  for all $\lambda>0$ we have

\[
{\rm Var}_{\gamma_n}(e^{\lambda F})\leq \frac{C\lambda^2b^2}{\log \bigg(e+\frac{b^2}{aA}\bigg)}\E_{\gamma_n}\big[e^{2\lambda F}\big].
\]

\noindent In particular, for any $t\geq 0$,

\begin{equation}\label{eq.paou.valettas}
\p\bigg( \big|f(X)-\E[f(X)]|\big|\geq t\bigg)\leq 4 \exp\bigg(-c\max\bigg\{\frac{t^2}{b^2},\frac{t}{b}\sqrt{\log \big(e+\frac{b^2}{aA}\big)}\bigg\}\bigg)
\end{equation}

\noindent where $C,c> 0$ are universal constants.
\end{prop}

\begin{rem}
It is a simple matter to check that equation \eqref{eq.paou.valettas} is sharp (except for the left tail) for the function $f(x)=\max_{i=1,\ldots,n}x_i$. Such achievements are part of the Superconcentration phenomenon introduced by Chatterjee in \cite{Chatt1}. We also refer to \cite{KT, KT1, KT2} for recent results in this topic (in particular, the article \cite{KT} gives some kind of extension of Proposition \ref{prop.paou.valettas} for correlated Gaussian measures).
\end{rem}

Since Paouris and Valettas's work relies on Talagrand inequality \eqref{eq.talagrand1}, we wonder if Theorem \ref{thm.talagrand.ordre2.gaussien} can be of any help to precise any further the Concentration of Measure phenomenon for the Gaussian measure $\gamma_n$.

\subsubsection{Higher order of concentration of measure}
Recently, Bobkov, Gotze and Sambale wrote an article \cite{BobGoSam} about higher order of concentration inequalities. In particular, they studied sharpened forms of the Concentration of Measure phenomenon for functions typically centered at stochastic expansions (the so-called Hoeffdding decomposition) of order $d-1$ for any $d\in\N$. They obtained deviations for smooth functions of independent random variables under some probability measure $\nu$ satisfying a logarithmic Sobolev inequality. One of their main results involved some bounds of derivatives of order $d$. As a sample, they proved the following. \\

 As it is presented in \cite{BobGoSam}, some notations are needed. Given a function $f\in C^{d}(\R^n)$ we define $f^{(d)}$ to be the (hyper-) matrix whose entries

\[
f^{(d)}_{i_1\ldots i_d}(x) = \partial_{i_1\ldots i_d}f,\quad d=1,2,\ldots
\]

\noindent represent the $d$-fold (continuous) partial derivatives of $f$ at $x\in\R^n$. By considering $f^{(d)}(x)$ as a symmetric multilinear $d$-form, we define operator-type norms by

\[
|f^{(d)}(x)|_{Op}=\sup\{f^{(d)}(x)[v_1,\ldots,v_d]\,:\,|v_1|=\ldots=|v_d|=1\}
\]

\noindent For instance, $|f^{(1)}(x)|_{Op}$ is the Euclidean norm of the gradient $\nabla f(x)$, and $|f^{(2)}(x)|_{Op}$ is the operator norm of the Hessian $\nabla^2f(x)$. Furthermore, the following short-hand notation will be used

\[ 
\|f^{(d)}\|_{Op,p} =\bigg(\int_{\R^n}|f^{(d)}|^p_{Op}d\nu\bigg)^{1/p},\quad\text{for any}\quad p\in(0,+\infty].
\]

Now, we can state their result.
\begin{thm}[Bobkov-G\"otze-Sambale]\label{thm.bobkov.gotze.sambale}
 Let $\nu$ be a probability measure on $\R^n$ satisfying a logarithmic Sobolev inequality with constant $\sigma^2$ and let $f\,:\,\R^n\to \R$ be $C^{d}$-smooth function such that
 
 \[ 
 \int_{\R^n}fd\nu=0\quad \text{and}\quad \int_{\R^n}\partial_{i_1\ldots i_k}fd\nu=0
\]

\noindent for all $k = 1,\ldots, d-1$ and $1\leq i_1\leq \ldots\leq i_k\leq n$. Assume that

\[
\|f^{(d)}\|_{HS,2}\leq 1\quad \text{and}\quad \|f^{(d)}\|_{Op,\infty}\leq 1
\]
Then, there exists some universal constant $c > 0$ such that

\[
\int_{\R^n}\exp\bigg(\frac{c}{\sigma^2}|f|^{2/d}\bigg)d\nu\leq 2.
\]
\end{thm}

\begin{rem} A possible choice is $c= 1/(8e)$. Note that, by integration by parts, if $\mu$ is the standard Gaussian measure $\gamma_n$, the conditions $\int_{\R^n}fd\nu=0$ and $\int_{\R^n}\partial_{i_1\ldots i_k}fd\nu=0$ are satisfied, if $f$ is orthogonal to all polynomials of (total) degree at most $d-1$. Such concentration's results for non Lipschitz functions (which are orthogonal to some part of an orthonormal basis) have been already obtained in various papers, we refer to the article \cite{BobGoSam} and references therein for more details.
\end{rem}

Their proof relies on the logarithmic Sobolev inequality together with some comparison of moments. Recall that logarithmic Sobolev's inequality is equivalent to the hypercontractive property of the associated semigroup (cf. \cite{BGL}). We ask if it is possible to recover their results with semigroup arguments ? In particular, is it possible to prove (and maybe improve by a dimension factor) Theorem \ref{thm.bobkov.gotze.sambale} (for $d=2$) with Talagrand inequality at order two from Theorem \ref{thm.talagrand.ordre2.gaussien} ?

\subsection{Gaussian influences}

In \cite{NKS}, the authors extended the notion of influence \ref{eq.influence} to a continuous setting. This notion has also been investigated in \cite{CoLed} (cf. Theorem $6$, p.$15$). This particular theorem relies on a variation on Talagrand inequality. In a Gaussian context, they obtained the following result 

\begin{thm}[Cordero-Erausquin, Ledoux]\label{thm.coled.2}
Let $f\,:\,\R^n\to\R$ be a smooth function such that $|f|\leq 1$, then 

\[
{\rm Var}_{\gamma_n}(f)\leq C\sum_{i=1}^n\frac{\|\partial_i f\|_1(1+\|\partial_i f\|_1)}{\big[1+\log^+(\frac{1}{\|\partial_if\|_1})\big]^{1/2}}
\]
for some universal constant $C>0$.
\end{thm}

This inequality is of particular interest when $f=1_A$ (or some smooth approximation) for some subset $A$ in $\R^n$. Indeed, $\|\partial_i f\|_1$ can be seen as the geometric influence $I_i(A)$ of the $i$-th coordinate on the set $A$ and, if $\gamma_n(A)=a$, it can be proved (cf. Corollary $7$, p.$17$ in \cite{CoLed}) that 

\[
I_i(A)\geq \frac{a(1-a)\log n^{1/2}}{Cn}.
\]
As observed by Bouyrie (cf. \cite{Boubou}), it is natural to ask if some variations around Theorem \ref{thm.talagrand.ordre2.gaussien} can be of any help to precise the last inequality for some subset $A$. Indeed, Bouyrie noticed that the combination of the arguments presented in \cite{CoLed} (during the proof of Theorem \ref{thm.coled.2}) and Talagrand inequality (of order $2$) \ref{thm.talagrand.ordre2.gaussien} yields the following inequality : let $f\,:\,\R^n\to\R$ be smooth enough such that $|f|\leq 1$ then

\[
{\rm Var}_{\gamma_n}(f)-\bigg|\int_{\R^n}\nabla fd\gamma_n\bigg|^2\leq 8\sum_{i,j=1}^n\frac{\|\partial_{ij} f\|_1}{1+\log (1/\|\partial_{ij}f\|_1)}.
\]

In particular, when $f$ is a smooth approximation of $1_A$ (with $A\subset\R^n$), notice that in this case, by integrations by parts, the left-hand-side corresponds to
\[
\gamma_n(A)\big(1-\gamma_n(A)\big)-b(A)
\]
where $b(A)$ designs the barycenter of A defined as $\int_{A}xd\gamma_n(x)$. However, the right hand side seems more complicated to interpret geometrically.  We wonder if it can be of any significance if $A$ is chosen to be a half-space.

\bigskip

\textit{Acknowledgment. This work has been initiated during my thesis and I thank my Ph.D advisor M. Ledoux for introducing this problem to me and for fruitful discussions. I am also indebted to K. Oleszkiewicz for several comments and precious advices. I also want to thank C. Houdr\'e for kindly pointing out to me the reference \cite{HPAS} and R. Kumolka for his help with the linguistic. Finally, I warmly thank the anonymous referee and R. Bouyrie for helpful comments in improving the exposition.}\\

\bibliographystyle{plain}

\end{document}